\documentclass{amsart}
\usepackage{mathrsfs}

\theoremstyle{plain}
\newtheorem{thm}{Theorem}[section]
\newtheorem{lemma}[thm]{Lemma}

\newtheorem{prop}[thm]{Proposition}

\theoremstyle{definition}

\newtheorem{remark}[thm]{Remark}

\newtheorem{example}[thm]{Example}
\newtheorem{Assumption}[thm]{Assumption}

\catcode`\@=11
\def\mequal{\mathrel{\mathpalette\@mvereq{\hbox{\sevenrm m}}}}
\def\@mvereq#1#2{\lower.5\p@\vbox{\baselineskip\z@skip\lineskip1.5\p@
    \ialign{$\m@th#1\hfil##\hfil$\crcr#2\crcr=\crcr}}}
\def\partr#1#2{/\kern-.08333em/_{#1,#2}^{\phantom{.}}}
\def\invpartr#1#2{/\kern-.08333em/_{#1,#2}^{-1}}
\def\hpartr#1#2{/\kern-.08333em/_{#1,#2}^{h}}
\def\Epartr#1#2{/\kern-.08333em/_{#1,#2}^{E}}

\def\newdot{{\kern.8pt\cdot\kern.8pt}}
\def\,{\relax\ifmmode\mskip\thinmuskip\else\thinspace\fi}
\def\{{\relax\ifmmode\lbrace\else $\lbrace$\fi}
\def\}{\relax\ifmmode\rbrace\else $\rbrace$\fi}

\font\sevenrm=cmr7

\newcommand\E{\mathbb{E}}
\newcommand\DD{\mathbb{D}}
\newcommand\EE{\mathbb{E}}

\newcommand\RR{\mathbb{R}}

\newcommand\PP{\mathbb{P}}

\newcommand{\SB}{{\mathscr B}}

\newcommand{\SF}{{\mathscr F}}

\newcommand{\CQFD}{\hfill $\square$}

\def\mathpal#1{\mathop{\mathchoice{\text{\rm #1}}%
   {\text{\rm #1}}{\text{\rm #1}}%
   {\text{\rm #1}}}\nolimits}

\def\cotanh{\mathpal{cotanh}}

\def\grad{\mathpal{grad}}
\def\id{\mathpal{id}}

\def\grad{\mathop{\rm grad}\nolimits}
\def\di{\displaystyle}
\def\f{\frac}
\def\a{\alpha }
\def\b{\beta }

\def\d{\delta }
\def\e{\varepsilon }
\def\G{\Gamma }
\def\g{\gamma }
\def\l{\lambda }
\def\n{\nabla }

\def\s{\sigma }

\begin{document}

\title[]{Stochastic algorithms for computing means of probability measures}

\author[M. Arnaudon]{Marc Arnaudon} \address{Laboratoire de Math\'ematiques et
  Applications\hfill\break\indent CNRS: UMR 6086\hfill\break\indent
  Universit\'e de Poitiers, T\'el\'eport 2 - BP 30179\hfill\break\indent
  F--86962 Futuroscope Chasseneuil Cedex, France}
\email{marc.arnaudon@math.univ-poitiers.fr}

\author[C. Dombry]{Cl\'ement Dombry} \address{Laboratoire de Math\'ematiques et
  Applications\hfill\break\indent CNRS: UMR 6086\hfill\break\indent
  Universit\'e de Poitiers, T\'el\'eport 2 - BP 30179\hfill\break\indent
  F--86962 Futuroscope Chasseneuil Cedex, France}
\email{clement.dombry@math.univ-poitiers.fr}

\author[A. Phan]{Anthony Phan} \address{Laboratoire de Math\'ematiques et
  Applications\hfill\break\indent CNRS: UMR 6086\hfill\break\indent
  Universit\'e de Poitiers, T\'el\'eport 2 - BP 30179\hfill\break\indent
  F--86962 Futuroscope Chasseneuil Cedex, France}
\email{anthony.phan@math.univ-poitiers.fr}

\author[L. Yang]{Le Yang} \address{Laboratoire de Math\'ematiques et
  Applications\hfill\break\indent CNRS: UMR 6086\hfill\break\indent
  Universit\'e de Poitiers, T\'el\'eport 2 - BP 30179\hfill\break\indent
  F--86962 Futuroscope Chasseneuil Cedex, France}
\email{le.yang@math.univ-poitiers.fr}

%\keywords{}

%%%%%%%%%%%%%%%%%%%%%%%%%%%%%%%%%%%%%%%%%%%%%%%%%%%%%%%%%%%%%%%%%%%%%%%%%%
%
%  Abstract, Keywords, AMS classification
%
%%%%%%%%%%%%%%%%%%%%%%%%%%%%%%%%%%%%%%%%%%%%%%%%%%%%%%%%%%%%%%%%%%%%%%%%%%

\begin{abstract}\noindent
Consider a probability measure $\mu$ supported by a regular geodesic ball in a manifold. For any $p\ge 1$ we define a stochastic algorithm which converges almost surely to the $p$-mean $e_p$ of  $\mu$. Assuming furthermore that the functional to minimize is regular around $e_p$, we prove that a natural renormalization of the inhomogeneous Markov chain converges in law into an inhomogeneous diffusion process. We give an explicit expression of this process, as well as its local characteristic.
\end{abstract}

\maketitle
\tableofcontents

%%%%%%%%%%%%%%%%%%%%%%%%%%%%%%%%%%%%%%%%%%%%%%%%%%%%%%%%%%%%%%%%%%%%%%%%%%%%
%
%  Actual Body of the Paper
%
%%%%%%%%%%%%%%%%%%%%%%%%%%%%%%%%%%%%%%%%%%%%%%%%%%%%%%%%%%%%%%%%%%%%%%%%%%%%

%%%%%%%%%%%%%%%%%%%%%%%%%%%%%%%%%%%%%%%%%%%%%%%%%%%%%%%%%%%%%%%%%%%%%%%%%%%%

\section{Introduction}\label{Section1}
The geometric barycenter of a set of points is the point which minimizes the sum of the distances at the power~$2$ to these points. It is the most common estimator is statistics, however it is sensitive to outliers, and it is natural to replace power~$2$ by $p$ for some $p\in [1,2)$, which leads to the definition of $p$-mean. When $p=1$, the minimizer is the median of the set of points, very often used in robust statistics. In many applications, $p$-means with some  $p\in (1,2)$ give the best compromise.

The Fermat-Weber problem concerns finding the median $e_1$ of a set of points in an Euclidean space.  Numerous authors worked out algorithms for computing~$e_1$. The first algorithm was proposed by Weiszfeld in~\cite{Weiszfeld:37}. It has been extended to sufficiently small domains in Riemannian manifolds with nonnegative curvature by Fletcher and al in~\cite{Fletcher:09}. A complete generalization to manifolds with positive or negative curvature, including existence and uniqueness results (under some convexity conditions in positive curvature), has been given by one of the authors in~\cite{Yang:10}. 

The Riemannian barycenter or Karcher mean of a set of points in a manifold or more generally of a probability measure has been extensively studied, see e.g. \cite{Karcher:77}, \cite{Kendall:90}, \cite{Kendall:91}, \cite{Emery-Mokobodzki:91}, \cite{Picard:94}, \cite{Arnaudon-Li:05}, where questions of existence, uniqueness, stability, relation with martingales in manifolds, behaviour  when measures are pushed by stochastic flows have been considered. The Riemannian barycenter corresponds to $p=2$ in the above description. Computation of Riemannian barycenters by gradient descent has been performed by Le in \cite{Le:04}.

 In \cite{Afsari:10} Afsari proved existence and uniqueness of $p$-means, $p\ge 1$ on geodesic balls with radius $\di r<\f12\min\left\{{\rm inj}(M),\f{\pi}{2\a}\right\}$ if $p\in [1,2)$, and $\di r<\f12\min\left\{{\rm inj}(M),\f{\pi}{\a}\right\}$ if $p\ge 2$. Here ${\rm inj}(M)$ is the injectivity radius of $M$ and $\a>0$ is such that the sectional curvatures in $M$ are bounded above by $\a^2$. The point is that in the case $p\ge 2$, the functional to minimize is not convex any more, which makes the situation much more difficult to handle.

In this paper, under the assumptions of~\cite{Afsari:10} we provide in Theorem~\ref{2.P1} stochastic algorithms which converge almost surely to $p$-means in manifolds, which are easier to implement than gradient descent algorithm since computing the gradient of the function to minimize is not needed. The idea is at each step to go in the direction of a point of the support of $\mu$. The point is chosen at random according to $\mu$ and the size of the step is a well chosen function of the distance to the point, $p$ and the number of the step. For general convergence results on recursive stochastic algorithms, see  \cite{Ljung:77} Theorem~1. However they do not cover the manifold case and nonlinearity of geodesics.  Here we give a  proof using martingale convergence theorem, and the main point consists in determining and estimating all the geometric quantities, checking that under our curvature conditions all the convergence assumptions are fulfilled, since our processes live in manifolds. See also~\cite{Benveniste-Goursat-Ruget:80} for convergence in probability of recursive algorithms.

  The speed of convergence is studied, and in theorem~\ref{T1} we prove that the renormalized inhomogeneous Markov chain of Theorem~\ref{2.P1} converges in law to an inhomogeneous diffusion process. This is an invariance principle type result, see e.g. \cite{KhasMinskii:66}, \cite{McLeish:74}, \cite{Berger:97}, \cite{Gouezel:10} for related works. Here again the main point is to obtain the characteristics of the limiting process from the curvature conditions, the conditions on the support of the mesure and estimates on Jacobi fields. Moreover we consider convergence in law for the Skorohod topology, and the limit depends in a crucial way on the decreasing steps of the algorithms.

\section{Results}
\subsection{$p$-means in regular geodesic balls}\label{Section2}
\setcounter{equation}0

Let $M$ be a Riemannian manifold with pinched sectional curvatures. Let $\a,\b>0$ such that $\a^2$ is a positive upper bound for sectional curvatures on $M$, and $-\b^2$ is a negative lower bound for sectional curvatures on $M$.
Denote by $\rho$ the Riemannian distance on $M$.

In $M$ consider a geodesic ball $B(a,r)$ with $a\in M$. 
Let $\mu$ be a probability measure with support included in a compact convex subset $K_\mu$ of $B(a,r)$. Fix $p\in[1,\infty)$.
We will always make the following assumptions on $(r,p,\mu)$:
\begin{Assumption}
\label{A1}
The support of $\mu$ is not reduced to one point.
Either $p>1$ or the support of $\mu$ is
 not contained in a line, and the radius $r$ satisfies
\begin{equation}
 \label{rp}
r<r_{\a,p}\quad \hbox{with}\left\{\begin{array}{ccc}
                              r_{\a,p}&=\f12\min\left\{{\rm inj}(M),\f{\pi}{2\a}\right\} &\hbox{if}\  p\in [1,2)\\
r_{\a,p}&=\f12\min\left\{{\rm inj}(M),\f{\pi}{\a}\right\} &\hbox{if} \ p\in [2,\infty)
                             \end{array}
\right.
\end{equation}
 \end{Assumption}

Note that $B(a,r)$ is convex if $r<\f12\min\left\{{\rm inj}(M),\f{\pi}{\a}\right\}$.

Under assumption~\ref{A1}, it has been proved in \cite{Afsari:10} (Theorem~2.1) that the function
\begin{equation}
\label{2.1}
\begin{split}
H_p : M&\to \RR_+\\
x&\mapsto \int_M\rho^p(x,y)\mu(dy)
\end{split}
\end{equation}
has a unique minimizer $e_p$ in $M$, the $p$-mean of $\mu$, and moreover $e_p\in B(a,r)$.  If $p=1$, $e_1$ is the median of $\mu$.

It is easily checked that if $p\in [1,2)$, then $H_p$ is strictly convex on $B(a,r)$. On the other hand, if $p\ge 2$ then $H_p$ is of class $C^2$ on $B(a,r)$.

\begin{prop}
\label{2.C1}
Let $K$ be a convex subset of $B(a,r)$ containing the support of $\mu$. Then there exists $C_{p,\mu,K}>0$ such that for all $x\in K$,
\begin{equation}
\label{2.H3}
H_p(x)-H_p(e_p)\ge \f{C_{p,\mu,K}}{2}\rho(x,e_p)^2.
\end{equation}
Moreover if $p\ge 2$ then we can choose $C_{p,\mu,K}$ so that  for all $x\in K$,
\begin{equation}
\label{2.H4}
\|\grad_x H_p \|^2\ge C_{p,\mu,K}\left(H_p(x)-H_p(e_p)\right).
\end{equation}
\end{prop}

In the sequel, we fix
\begin{equation}
\label{K}
K=\bar B(a, r-\e)\quad \hbox{with}\quad \e= \f{\rho(K_\mu,B(a,r)^c)}{2}.
\end{equation}

We now state our main result: we define a stochastic gradient algorithm $(X_k)_{k\geq 0}$ to approximate the $p$-mean $e_p$ and prove its convergence.
\begin{thm}
\label{2.P1}
Let $(P_k)_{k\ge 1}$ be a sequence of independent $B(a,r)$-valued random variables, with law $\mu$. Let $(t_k)_{k\ge 1}$ be a sequence of positive numbers satisfying
\begin{equation}
\label{2.2bis} \forall k\ge 1,\ \
t_k\leq\min\left(\f1{C_{p,\mu,K}},
\f{\rho(K_\mu,B(a,r)^c)}{2p(2r)^{p-1}}\right),
\end{equation}
\begin{equation}
\label{2.2}
\sum_{k=1}^\infty t_k=+\infty\quad\hbox{and}\quad \sum_{k=1}^\infty t_k^2<\infty.
\end{equation}
Letting $x_0\in K$, define inductively the random walk $(X_k)_{k\ge 0}$ by
\begin{equation}
\label{2.3}
X_0=x_0\quad\hbox{and for $k\ge 0$}\quad X_{k+1}=\exp_{X_{k}}\left(-t_{k+1}\grad_{X_{k}} F_p(\cdot, P_{k+1})\right)
\end{equation}
where $F_p(x,y)=\rho^p(x,y)$, with the convention $\grad_x F_p(\cdot, x)=0$.

The random walk $(X_k)_{k\ge 1}$ converges in $L^2$ and almost surely to $e_p$.
\end{thm}

In the following example, we focus on the case $M=\RR^d$ and $p=2$ where drastic simplifications occur.
\begin{example}
\label{E1}
In the case when $M=\RR^d$ and $\mu$ is a compactly supported probability measure on $\RR^d$, the stochastic gradient algorithm \eqref{2.3} simplifies into
\[
X_0=x_0\quad\hbox{and for $k\ge 0$}\quad X_{k+1}=X_k-t_{k+1}\grad_{X_{k}} F_p(\cdot, P_{k+1}).
\]
If furthermore $p=2$, clearly $e_2=\EE[P_1]$ and $\grad_x F_p(\cdot, y)=2(x-y)$, so that the linear relation
\[
X_{k+1}=(1-2t_{k+1})X_k+2t_{k+1}P_{k+1},\quad k\geq 0
\]
holds true and an easy induction proves that
\begin{equation} \label{2.12}
X_k=x_0\prod_{j=0}^{k-1}(1-2t_{k-j})+2\sum_{j=0}^{k-1}P_{k-j}t_{k-j}\prod_{\ell=0}^{j-1}(1-2t_{k-\ell}),\quad k\ge 1.
\end{equation}
Now, taking $\di t_k=\f1{2k}$, we have
$$
\prod_{j=0}^{k-1}(1-2t_{k-j})=0\quad \hbox{and} \quad \prod_{\ell=0}^{j-1}(1-2t_{k-\ell})=\f{k-j}{k}
$$
so that
$$
X_k=\sum_{j=0}^{k-1}P_{k-j}\f1k=\f1k\sum_{j=1}^kP_j.
$$
The stochastic gradient algorithm estimating the mean $e_2$ of $\mu$ is given by the empirical mean of a growing sample of independent random variables with distribution $\mu$. In this simple case, the result of Theorem~\ref{2.P1} is nothing but the strong law of large numbers. Moreover, fluctuations around the mean are given by the central limit theorem and Donsker's theorem.
\end{example}

\subsection{Fluctuations of the stochastic gradient algorithm}\label{Section3}

The notations are the same as in the beginning of section~\ref{Section2}. We still make assumption~\ref{A1}.
Let us define $K$ and $\e$ as in \eqref{K} and let
\begin{equation}
 \label{3.0}
\d_1=\min\left(\f1{C_{p,\mu,K}},
\f{\rho(K_\mu,B(a,r)^c)}{2p(2r)^{p-1}}\right).
\end{equation}

We consider the time inhomogeneous $M$-valued Markov chain \eqref{2.3} in the particular case when
\begin{equation}\label{3.01}
t_k=\min\left(\frac{\delta}{k}, \delta_1\right), \quad k\geq 1
\end{equation}
for some $\delta>0$. The particular sequence $(t_k)_{k\geq 1}$ defined by \eqref{3.01} satisfies  \eqref{2.2bis} and \eqref{2.2}, so Theorem \ref{2.P1} holds true and the stochastic gradient algorithm $(X_k)_{k\ge 0}$ converges a.s. and in $L^2$ to the $p$-mean $e_p$.

%\begin{equation}\label{3.1}
%X_0=x_0,
%\end{equation}
%\begin{equation}
% \label{3.2bis}
%X_{k+1}=\exp_{X_k}\left(-\d_1\grad_{X_k}F_p(\cdot, P_{k+1})\right), \quad 0\le k< k_0,
%\end{equation}
%and
%\begin{equation}\label{3.2}
%X_{k+1}=\exp_{X_k}\left(-\f{\d}{k+1}\grad_{X_k}F_p(\cdot, P_{k+1})\right), \quad k\ge k_0,
%\end{equation}
%where $k_0$ is such that
%\begin{equation}
% \label{3.3}
%\f{\d}{k_0}<\d_1.
%\end{equation}
%We know By Proposition~\ref{2.P1} that .

In order to study the fluctuations around the $p$-mean $e_p$, we define for $n\ge 1$ the rescaled $T_{e_p}M$-valued Markov chain $(Y_k^n)_{k\ge 0}$ by
\begin{equation}
 \label{3.4}
Y_k^n=\f{k}{\sqrt{n}}\exp_{e_p}^{-1}X_k.
\end{equation}

We will prove convergence of the sequence of process $(Y_{[nt]}^n)_{t\ge 0}$ to a non-homogeneous diffusion process. The limit process is defined in the following proposition:

\begin{prop}\label{pr2}
Assume that $H_p$ is $C^2$ in a neighborhood of~$e_p$, and that  $\d>C_{p,\mu,K}^{-1}$. Define
\[
\Gamma=\EE\left[\grad_{e_p}F_p(\cdot,P_1)\otimes \grad_{e_p}F_p(\cdot, P_1)\right]
\]
and $G_\d(t)$ the generator
\begin{equation} \label{3.26bis}
 G_\d(t) f(y):= \langle d_yf,t^{-1}(y-\d\n dH_p(y,\cdot)^\sharp)\rangle+\f{\d^2}2{\rm Hess}_yf\left(\Gamma\right)
\end{equation}
where $\n dH_p(y,\cdot)^\sharp$ denotes the dual vector of the linear form $\n dH_p(y,\cdot)$.

There exists a unique inhomogeneous diffusion process $\di (y_\d(t))_{t> 0}$ on $T_{e_p}M$ with generator $G_\d(t)$ and converging in probability to $0$ as $t\to 0^+$.

The process $y_\d$ is continuous, converges a.s. to $0$ as $t\to 0^+$ and has the following integral representation:
\begin{equation}
\label{3.32}
 y_\d(t)=\sum_{i=1}^dt^{1-\d\l_i}\int_0^ts^{\d\l_i-1}\langle \d\s \,dB_s,e_i\rangle e_i,\quad t\ge
 0,
\end{equation}
where $B_t$ is a standard Brownian motion on $T_{e_p}M$,  $\s\in {\rm End}(T_{e_p}M)$ satisfies $\s\s^\ast=~\G$,
%(i.e. the symmetric nonnegative endomorphism such that $s^\ast\s=\G$).
$(e_i)_{1\le i\le d}$ is an orthonormal basis diagonalizing the symmetric bilinear form $\n dH_p(e_p)$ and $(\l_i)_{1\le i\le d}$ are the associated eigenvalues.
\end{prop}

Note that the integral representation \eqref{3.32} implies that $y_\d$ is the centered Gaussian process with covariance
\begin{equation}\label{eq:cov}
\EE\left[y^i_\d(t_1)y^j_\d(t_2)\right]=\frac{\d^2 \G (e_i^\ast\otimes e_j^\ast)} {\d(\lambda_i+\lambda_j)-1}{t_1^{1-\d\l_i}t_2^{1-\d\l_j}} {(t_1\wedge t_2)^{\delta(\l_i+\l_j)-1}},
\end{equation}

where $y_\d^i(t)=\langle y_\d(t),e_i\rangle$, $1\leq i,j\leq d$ and $t_1,t_2\geq 0$.

Our main result on the fluctuations of the stochastic gradient algorithm is the following:
\begin{thm}  \label{T1}
Assume that either $e_p$ does not belong to the support of $\mu$ or $p\ge 2$. Assume furthermore that  $\d>C_{p,\mu,K}^{-1}$.
The sequence of processes $\di \left(Y_{[nt]}^n\right)_{t\ge 0}$ weakly converges in $\DD((0,\infty),T_{e_p}M)$ to $y_\d$.
\end{thm}

\begin{remark}
The assumption on $e_p$ implies that $H_p$ is of class $C^2$ in a neighbourhood of $e_p$.
  In the case $p>1$, in the ``generic'' situation for applications, $\mu$ is a discrete measure and $e_p$ does not belong to its support. For $p=1$ one has to be more careful since if $\mu$ is equidistributed in a random set of points, then with positive probability $e_1$ belongs to the support of $\mu$.
\end{remark}

\begin{remark}
From section~\ref{Section2} we know that, when $p\in(1,2]$, the constant
$$\di C_{p,\mu,K}=p(2r)^{p-2}\left(\min\left(p-1,2\a r\cot\left(2\a r\right)\right)\right)$$
is explicit. The constraint $\d>C_{p,\mu,K}^{-1}$ can easily be checked in this case.
\end{remark}

\begin{remark}
In the case $M=\RR^d$, $Y_k^n=\frac{k}{\sqrt n}(X_k-e_p)$ and the
tangent space $T_{e_p}M$ is identified to $\RR^d$. Theorem \ref{T1}
holds and, in particular, when $t=1$, we obtain a central limit
Theorem: $\sqrt n(X_n-e_p)$ converges as $n\to\infty$ to a centered
Gaussian $d$-variate distribution (with covariance structure given
by \eqref{eq:cov} with $t_1=t_2=1$). This is a central limit
theorem: the fluctuations of the stochastic gradient algorithm are
of scale $n^{-1/2}$ and asymptotically Gaussian.
\end{remark}

\section{Proofs}
For simplicity, let us write shortly $e=e_p$ in the proofs.
\setcounter{equation}0
\subsection{Proof of Proposition  \ref{2.C1}}\ \\
For $p=1$ this is a direct consequence of~\cite{Yang:10} Theorem~3.7. 

Next we consider the case $p\in(1,2)$.

Let  $K\subset B(a,r)$ be a compact convex set containing the support of $\mu$. Let  $x\in K\backslash\{e\}$, $t=\rho(e,x)$,  $u\in T_eM$ the unit vector such that $\exp_e(\rho(e,x)u)=x$, and
$\g_{u}$  the geodesic with initial speed $u$ : $\dot\g_u(0)=u$. For $y\in K$, letting $h_y(s)=\rho(\g_u(s),y)^p$, $s\in [0,t]$, we have since $p>1$
$$
h_y(t)=h_y(0)+th_y'(0)+\int_0^t(t-s)h_y''(s)\,ds
$$
with the convention $h_y''(s)=0$ when $\g_u(s)=y$. Indeed, if $y\not\in \g([0,t])$ then $h_y$ is smooth, and if $y\in \g([0,t])$, say $y=\g(s_0)$ then $h_y(s)=|s-s_0|^p$ and the formula can easily be checked.

By standard calculation,
\begin{equation}
\label{2.E2bis}
\begin{split}
&h_y''(s)\\&\ge p\rho(\g_u(s),y)^{p-2}\\&\quad \times\left((p-1)\|\dot\g_u(s)^{T(y)}\|^2+\|\dot\g_u(s)^{N(y)}\|^2 \a\rho(\g_u(s),y)\cot\left(\a \rho(\g_u(s),y)\right)\right)
\end{split}
\end{equation}
with $\di \dot\g_u(s)^{T(y)}$ (resp. $\di \dot\g_u(s)^{N(y)}$) the tangential (resp. the normal) part of $\dot\g_u(s)$ with respect to $\di n(\g_u(s),y)=\f1{\rho(\g_u(s),y)}\exp_{\g_u(s)}^{-1}(y)$:
$$
\dot\g_u(s)^{T(y)}=\langle \dot\g_u(s), n(\g_u(s),y)\rangle n(\g_u(s),y),\quad  \dot\g_u(s)^{N(y)}=\dot\g_u(s)-\dot\g_u(s)^{T(y)}.
$$
From this we get
\begin{equation}
\label{2.E2ter}
\begin{split}
h_y''(s)\ge p\rho(\g_u(s),y)^{p-2}\left(\min\left(p-1,2\a r\cot\left(2\a r\right)\right)\right)
\end{split}.
\end{equation}
Now
\begin{align*}
&H_p(\g_u(t'))\\&=\int_{B(a,r)}h_y(\g_u(t'))\,\mu(dy)\\&=\int_{B(a,r)}h_y(0)\,\mu(dy)+t'\int_{B(a,r)}h_y'(0)\,\mu(dy)+\int_0^{t'}(t'-s)\left(\int_{B(a,r)}h_y(s)''\,\mu(dy)\right)\,ds
\end{align*}
and $H_p(\g_u(t'))$ attains its minimum at $t'=0$, so $\di \int_{B(a,r)}h_y'(0)\,\mu(dy)=0$ and  we have
$$
H_p(x)=H_p(\g_u(t))=H_p(e)+\int_0^t(t-s)\left(\int_{B(a,r)}h_y(s)''\,\mu(dy)\right)\,ds.
$$
Using Equation~\eqref{2.E2ter} we get
\begin{equation}
 \label{2.H1bis}
\begin{split}
&H_p(x)\ge H_p(e)\\&+\int_0^t\left((t-s)\int_{B(a,r)}p\rho(\g_u(s),y)^{p-2}\left(\min\left(p-1,2\a r\cot\left(2\a r\right)\right)\right)\,\mu(dy)\right)\,ds.
\end{split}
\end{equation}
Since   $p\le 2$  we have $\rho(\g_u(s),y)^{p-2}\ge (2r)^{p-2}$ and
\begin{equation}
 \label{2.H1ter}
\begin{split}
H_p(x)&\ge H_p(e)+\f{t^2}{2}p(2r)^{p-2}\left(\min\left(p-1,2\a r\cot\left(2\a r\right)\right)\right).
\end{split}
\end{equation}
So letting  $$C_{p,\mu,K}=p(2r)^{p-2}\left(\min\left(p-1,2\a r\cot\left(2\a r\right)\right)\right)$$ we obtain
\begin{equation}
\label{2.H1}
H_p(x)\ge H_p(e)+\f{C_{p,\mu,K}\rho(e,x)^2}{2}.
\end{equation}

To finish let us consider the case $p\ge 2$. 

In the proof of~\cite{Afsari:10} Theorem~2.1, it is shown that $e$ is the only zero of the maps $x\mapsto \grad_x H_p$ and $x\mapsto H_p(x)-H_p(e)$, and that $\n d H_p(e)$ is strictly positive. This implies that~\eqref{2.H3} and~\eqref{2.H4} hold on some neighbourhood $B(e,\e)$ of $e$. By compactness and the fact that $H_p-H_p(e)$ and $\grad H_p$ do not vanish on $K\backslash B(e,\e)$ and $H_p-H_p(e)$ is bounded, possibly modifying the constant $C_{p,\mu,K}$,~\eqref{2.H3} and~\eqref{2.H4} also holds on $K\backslash B(e,\e)$.

\CQFD

\subsection{Proof of Theorem \ref{2.P1}}\ \\
Note that, for $x\not=y$,
\[
\grad_x F(\cdot, y)=p\rho^{p-1}(x,y)\f{-\exp_x^{-1}(y)}{\rho(x,y)}=-p\rho^{p-1}(x,y)n(x,y),
\]
whith $\di n(x,y):=\f{\exp_x^{-1}(y)}{\rho(x,y)}$ a unit vector. So,
with the condition~\eqref{2.2bis} on $t_k$, the random walk
$(X_k)_{k\ge 0}$ cannot exit $K$: if $X_k\in K$ then there are two
possibilities for $X_{k+1}$:
\begin{itemize}
\item
 either $X_{k+1}$ is in the geodesic between $X_k$ and $P_{k+1}$ and belongs to $K$ by convexity of $K$;
\item or $X_{k+1}$ is after $P_{k+1}$, but since
\begin{align*}
\|t_{k+1}\grad_{X_{k}} F_p(\cdot, P_{k+1})\|&=t_{k+1}p\rho^{p-1}(X_k,P_{k+1})\\&\le \f{\rho(K_\mu,B(a,r)^c)}{2p(2r)^{p-1}}p\rho^{p-1}(X_k,P_{k+1})\\&\le \f{\rho(K_\mu,B(a,r)^c)}{2},
\end{align*}
we have in this case $$\rho(P_{k+1}, X_{k+1})\le \f{\rho(K_\mu,B(a,r)^c)}{2}$$ which implies that $X_{k+1}\in K.$
\end{itemize}

First consider the case $p\in [1,2)$.

For $k\ge 0$ let
$$
t\mapsto E(t):=\f12 \rho^2\left(e, \g(t)\right),
$$
$\di \g(t)_{ t\in [0, t_{k+1}]}$ the  geodesic satisfying $\di \dot
\g(0)=-\grad_{X_k}F_p(\cdot,P_{k+1})$. We have for all $t\in
[0,t_{k+1}]$
\begin{equation}
\label{2.E} E''(t)\le C(\b,r,p):=p^2(2r)^{2p-1}\b\cotanh (2\b r)
\end{equation}
(see e.g. \cite{Yang:10}).
By Taylor formula,
\begin{align*}
&\rho(X_{k+1},e)^2\\
&=2E(t_{k+1})\\
&=2E(0)+2t_{k+1}E'(0)+t_{k+1}^2E''(t)\quad \hbox{for some $t\in[0,t_{k+1}]$}\\
&\le \rho(X_{k},e)^2+2t_{k+1}\langle \grad_{X_{k}} F_p(\cdot, P_{k+1}), \exp_{X_k}^{-1}(e)\rangle +t_{k+1}^2C(\b,r,p).\\
\end{align*}
Now from the convexity of $x\mapsto F_p(x,y)$ we have for all $x,y\in B(a,r)$
\begin{equation}
\label{2.4}
F_p(e,y)-F_p(x,y)\ge \left\langle \grad_xF_p(\cdot ,y), \exp_x^{-1}(e)\right\rangle.
\end{equation}
This applied with $x=X_k$, $y=P_{k+1}$ yields
\begin{equation}
\label{2.5} \rho(X_{k+1},e)^2\le
\rho(X_{k},e)^2-2t_{k+1}\left(F_p(X_k,P_{k+1})-F_p(e,P_{k+1})\right)
+C(\b,r,p)t_{k+1}^2.
\end{equation}
Letting for $k\ge 0$ $\SF_k=\s(X_\ell, 0\le \ell\le k)$, we get
\begin{align*}
&\EE\left[\rho(X_{k+1},e)^2|\SF_k\right]\\&\le \rho(X_{k},e)^2-2t_{k+1}\int_{B(a,r)}\left(F_p(X_k,y)-F_p(e,y)\right)\mu(dy)+C(\b,r,p)t_{k+1}^2\\
&=\rho(X_{k},e)^2-2t_{k+1}\left(H_p(X_k)-H_p(e)\right)+C(\b,r,p)t_{k+1}^2\\
&\le \rho(X_{k},e)^2+C(\b,r,p)t_{k+1}^2
\end{align*}
so that the process $(Y_k)_{k\ge 0}$ defined by
\begin{equation}
\label{2.6} Y_0=\rho(X_0,e)^2\quad \hbox{and for $k\ge 1$}\quad
Y_{k}=\rho(X_{k},e)^2-C(\b,r,p)\sum_{j=1}^kt_j^2
\end{equation}
is a bounded  supermartingale. So it converges in $L^1$ and almost surely. Consequently $\rho(X_{k},e)^2$
also converges in $L^1$ and almost surely.

Let
\begin{equation}
\label{2.7}
a=\lim_{k\to\infty}\EE\left[\rho(X_{k},e)^2\right].
\end{equation}
We want to prove that $a=0$. We already proved that
\begin{equation}
\label{2.8} \EE\left[\rho(X_{k+1},e)^2|\SF_k\right]\le
\rho(X_{k},e)^2-2t_{k+1}\left(H_p(X_k)-H_p(e)\right)+C(\b,r,p)t_{k+1}^2.
\end{equation}
Taking the expectation and using Proposition~\ref{2.C1}, we obtain
\begin{equation}
\label{2.9} \EE\left[\rho(X_{k+1},e)^2\right]\le
\EE\left[\rho(X_{k},e)^2\right]-t_{k+1}C_{p,\mu,K}\EE\left[\rho(X_{k},e)^2\right]+C(\b,r,p)t_{k+1}^2.
\end{equation}
An easy induction proves that for $\ell\ge 1$,
\begin{equation}
\label{2.10} \EE\left[\rho(X_{k+\ell},e)^2\right]\le
\prod_{j=1}^\ell
(1-C_{p,\mu,K}t_{k+j})\EE\left[\rho(X_{k},e)^2\right]+C(\b,r,p)\sum_{j=1}^\ell
t_{k+j}^2.
\end{equation}
Letting $\ell\to \infty$ and using the fact that $\sum_{j=1}^\infty t_{k+j}=\infty$ which implies $$\prod_{j=1}^\infty (1-C_{p,\mu,K}t_{k+j})=0,$$ we get
\begin{equation}
\label{2.11} a\le C(\b,r,p)\sum_{j=1}^\infty t_{k+j}^2.
\end{equation}
Finally using $\sum_{j=1}^\infty t_{j}^2<\infty$ we obtain that  $ \lim_{k\to\infty}\sum_{j=1}^\infty t_{k+j}^2=0$, so $a=0$. This proves $L^2$ and almost sure convergence.

Next assume that $p\ge 2$.

For $k\ge 0$ let
$$
t\mapsto E_p(t):=H_p(\g(t)),
$$
$\di \g(t)_{ t\in [0, t_{k+1}]}$ the  geodesic satisfying $\di \dot
\g(0)=-\grad_{X_k}F_p(\cdot,P_{k+1})$. With a calculation similar to~\eqref{2.E} we get for all $t\in
[0,t_{k+1}]$
\begin{equation}
\label{2.E2} E_p''(t)\le 2C(\b,r,p):=\f{p^3}{2}(2r)^{3p-4}\left(2r\b\cotanh (2\b r)+2p-4\right).
\end{equation}
(see e.g. \cite{Yang:10}).
By Taylor formula,
\begin{align*}
H_p(X_{k+1})&=E_p(t_{k+1})\\
&=E_p(0)+t_{k+1}E_p'(0)+\f{t_{k+1}^2}{2}E_p''(t)\quad \hbox{for some $t\in[0,t_{k+1}]$}\\
&\le H_p(X_{k})+t_{k+1}\langle d_{X_k}H_p, \grad_{X_{k}} F_p(\cdot, P_{k+1})\rangle +t_{k+1}^2C(\b,r,p).\\
\end{align*}

We get
\begin{align*}
&\EE\left[H_p(X_{k+1})|\SF_k\right]\\&\le H_p(X_{k})-t_{k+1}\left\langle d_{X_k}H_p,\int_{B(a,r)}\grad_{X_{k}} F_p(\cdot, y)\mu(dy)\right\rangle+C(\b,r,p)t_{k+1}^2\\
&= H_p(X_{k})-t_{k+1}\left\langle d_{X_k}H_p,\grad_{X_{k}} H_p(\cdot)\right\rangle+C(\b,r,p)t_{k+1}^2\\
&= H_p(X_{k})-t_{k+1}\left\|\grad_{X_{k}} H_p(\cdot)\right\|^2+C(\b,r,p)t_{k+1}^2\\
&\le H_p(X_{k})-C_{p,\mu,K}t_{k+1}\left(H_p(X_k)-H_p(e)\right)+C(\b,r,p)t_{k+1}^2\\
\end{align*}
(by Proposition~\ref{2.1})
so that the process $(Y_k)_{k\ge 0}$ defined by
\begin{equation}
\label{2.6bis} Y_0=H_p(X_0)-H_p(e)\quad \hbox{and for $k\ge 1$}\quad
Y_{k}=H_p(X_{k})-H_p(e)-C(\b,r,p)\sum_{j=1}^kt_j^2
\end{equation}
is a bounded  supermartingale. So it converges in $L^1$ and almost surely. Consequently $H_p(X_k)-H_p(e)$
also converges in $L^1$ and almost surely.

Let
\begin{equation}
\label{2.7bis}
a=\lim_{k\to\infty}\EE\left[H_p(X_k)-H_p(e)\right].
\end{equation}
We want to prove that $a=0$. We already proved that
\begin{equation}
\label{2.8bis} 
\begin{split}&\EE\left[H_p(X_{k+1})-H_p(e)|\SF_k\right]\\&\le
H_p(X_k)-H_p(e)-C_{p,\mu,K}t_{k+1}\left(H_p(X_k)-H_p(e)\right)+C(\b,r,p)t_{k+1}^2.
\end{split}
\end{equation}
Taking the expectation we obtain
\begin{equation}
\label{2.9bis} \EE\left[H_p(X_{k+1})-H_p(e)\right]\le
(1-t_{k+1}C_{p,\mu,K})\EE\left[H_p(X_k)-H_p(e)\right]+C(\b,r,p)t_{k+1}^2
\end{equation}
so that proving that $a=0$ is similar to the previous case.

Finally~\eqref{2.H3} proves that $\rho(X_{k},e)^2$ converges in $L^1$ and almost surely to~$0$.
\CQFD

\subsection{Proof of Proposition \ref{pr2}}
%\begin{lemma}
% \label{L3.5}
%Assume $\d>C_{p,\mu,K}$.
%There exists a unique continuous inhomogeneous diffusion process $(y(t))_{0< t\le 1}$ on $T_eM$ with generator on $(0,1]$
%\begin{equation}
% \label{3.27}
%G_\d(t)(y)=\f{y}{t}-\d\n_edH(\f{y}{t},\cdot)^\sharp\rangle+\f12\EE\left[\grad_eF_p(\cdot,P_1)\otimes \grad_eF_p(\cdot, P_1)\right],
%\end{equation}
%and converging in $L^2$ to $0$ as $t\to 0$.

%Moreover this solution converges a.s. to $0$ as $t\to 0$.
%\end{lemma}
%\begin{proof}

Fix $\e>0$. Any diffusion process on $[\e,\infty)$ with generator $G_\d(t)$ is solution of a sde of the type
\begin{equation}
 \label{3.28}
dy_t=\f1tL_\d(y_t)\,dt+\d \s\,dB_t
\end{equation}
where $L_\d(y)=y-\d\n dH_p(y,\cdot)^\sharp$ and $B_t$ and $\s$ are
as in Proposition \ref{pr2}. This sde can be solved explicitely on
$[\e,\infty)$. The symmetric endomorphism $y\mapsto
\n dH_p(y,\cdot)^\sharp$ is diagonalisable in the orthonormal basis
$(e_i)_{1\le i\le d}$ with eigenvalues $(\l_i)_{1\le i\le d}$. The
endomorphism $L_\d=\id-\d \n dH_p(e)(\id,\cdot)^{\sharp}$ is also
diagonalisable in this basis with eigenvalues $(1-\d\l_i)_{1\le i\le
d}$. The solution $\di y_t=\sum_{i=1}^d y_t^i e_i$ of~\eqref{3.28}
started at $\di y_\e=\sum_{i=1}^d y_\e^i e_i$ is given by
\begin{equation}
 \label{3.29}
y_t=\sum_{i=1}^d\left(y_\e^i\e^{\d\l_i-1}+\int_\e^ts^{\d\l_i-1}\langle \d\s \,dB_s,e_i\rangle\right)t^{1-\d\l_i}e_i,\quad t\ge \e.
\end{equation}
Now by definition of $C_{p,\mu,K}$ we clearly have
\begin{equation}
 \label{3.30}
C_{p,\mu,K}\le \min_{1\le i\le d} \l_i.
\end{equation}
So the condition $\d C_{p,\mu,K}>1$ implies that  for all $i$, $\d\l_i-1>0$, and as $\e\to 0$,
\begin{equation}
 \label{3.31}
\int_\e^ts^{\d\l_i-1}\langle \d\s\,dB_s,e_i\rangle \to \int_0^ts^{\d\l_i-1}\langle \d\s\,dB_s,e_i\rangle\quad \mathrm{in\ probability.}
\end{equation}
Assume that a continuous  solution $y_t$  converging in probability to $0$ as $t\to 0^+$ exists. Since $y_\e^i\e^{\d\l_i-1}\to 0$ in probability as $\e\to 0$, we necessarily have using \eqref{3.31}
\begin{equation}\label{3.32bis}
 y_t=\sum_{i=1}^dt^{1-\d\l_i}\int_0^ts^{\d\l_i-1}\langle \d\s \,dB_s,e_i\rangle e_i,\quad t\ge 0.
\end{equation}
Note $y_\d^i$ is Gaussian with variance $\di \f{t\d^2\G(e_i^\ast\otimes e_i^\ast)}{2\d \l_i-1}$, so it converges in $L^2$ to $0$ as $t\to 0$.
Conversely, it is easy to check that equation~\eqref{3.32bis} defines a solution to~\eqref{3.28}.

To prove the a.s. convergence to $0$ we use the representation
$$
\int_0^ts^{\d\l_i-1}\langle \d\s\,dB_s,e_i\rangle=B_{\varphi_i(t)}^i
$$
where $B_s^i$ is a Brownian motion and $\di \varphi_i(t)=\f{\d^2\G(e_i^\ast\otimes e_i^\ast)}{2\d \l_i-1}t^{2\d\l_i-1}$. Then by the law of iterated logarithm
$$
\limsup_{t\downarrow 0}t^{1-\d\l_i}B_{\varphi_i(t)}^i\le \limsup_{t\downarrow 0}t^{1-\d\l_i}\sqrt{2\varphi_i(t)\ln\ln\left(\varphi_i^{-1}(t)\right)}
$$
But for $t$ small we have
$$
\sqrt{2\varphi_i(t)\ln\ln\left(\varphi_i^{-1}(t)\right)}\le t^{\d\l_i-3/4}
$$
so
$$
\limsup_{t\downarrow 0}t^{1-\d\l_i}B_{\varphi_i(t)}^i\le \lim_{t\downarrow 0}t^{1/4}=0.
$$
This proves a.s. convergence to $0$. Continuity is easily checked using the integral representation \eqref{3.32bis}.
\CQFD

\subsection{Proof of Theorem \ref{T1}}
Consider the time homogeneous Markov chain $(Z_k^n)_{k\ge 0}$ with
state space $[0,\infty)\times T_eM$ defined by
\begin{equation}
 \label{3.5}
Z_k^n=\left(\f{k}n,Y_k^n\right).
\end{equation}
The first component has a deterministic evolution and will be denoted by $t_k^n$; it satisfies
\begin{equation}
t_{k+1}^n=t_k^n+\f1n,\quad k\geq 0.
\end{equation}
Let $k_0$ be such that
\begin{equation}
 \frac{\d}{k_0}<\delta_1.
\end{equation}
Using equations \eqref{2.3}, \eqref{3.4} and \eqref{3.01}, we have for $k\geq k_0$,
\begin{equation}
Y_{k+1}^n= \f{nt_k^n+1}{\sqrt{n}}\exp_{e}^{-1}\left(\exp_{\exp_e\f1{\sqrt{n}t_k^n}Y_k^n}\left(-\f{\d}{nt_k^n+1}\grad_{\f1{\sqrt{n}t_k^n}Y_k^n}F_p(\cdot,P_{k+1})\right)\right).
\end{equation}

%\begin{equation}
% \label{3.6}
%Z_{k_0}^n=\left(t_{k_0}^{n},Y_{k_0}^n\right)
%\end{equation}
%and $Z_{k+1}^n=(t_{k+1}^n,Y_{k+1}^n)$ with for $k\ge k_0$,
%\begin{equation}
% \label{3.7}
%\left\{
%\begin{array}{ccc}
% t_{k+1}^n&=&t_k^n+\f1n\\
%Y_{k+1}^n&=& \f{nt_k^n+1}{\sqrt{n}}\exp_{e}^{-1}\left(\exp_{\exp_e\f1{\sqrt{n}t_k^n}Y_k^n}\left(-\f{\d}{nt_k^n+1}\grad_{\f1{\sqrt{n}t_k^n}Y_k^n}F_p(\cdot,P_{k+1})\right)\right)
%\end{array}
%\right..
%\end{equation}

Consider the transition kernel $P^n(z,dz')$ on $(0,\infty)\times
T_eM$ defined for $z=(t,y)$ by
\begin{equation}
 \label{3.8}\begin{split}
&P^n(z,A)=\\&\PP\left[\left(t+\f1n,\f{nt+1}{\sqrt{n}}\exp_{e}^{-1}\left(\exp_{\exp_e\f1{\sqrt{n}t}y}\left(-\f{\d}{nt+1}\grad_{\exp_{e}\f1{\sqrt{n}t}y}F_p(\cdot,P_{1})\right)\right)\right)\in A\right]
\end{split}
\end{equation}
where $A\in \SB((0,\infty)\times T_eM)$. Clearly this transition kernel drives the evolution of the Markov chain $(Z_k^n)_{k\ge k_0}$.

For the sake of clarity, we divide the proof of Theorem \ref{T1} into four lemmas.
\begin{lemma}
 \label{L3.2}
Assume that either $p\ge 2$ or $e$ does not belong to the support ${\rm supp}(\mu)$ of $\mu$ (note this implies that
for all $x\in {\rm supp}(\mu)$ the function $F_p(\cdot, x)$ is of class $C^2$ in a neighbourhood of $e$). Fix $\d>0$.
 Let $B$ be a bounded set in $T_eM$ and let $0<\e<T$. We have for all $C^2$ function $f$ on $T_eM$
\begin{equation}
 \label{3.18}
\begin{split}
 &n\left(f\left(\f{nt+1}{\sqrt n}\exp_{e}^{-1}\left(\exp_{\exp_e\f1{\sqrt{n}t}y}\left(-\f{\d}{nt+1}\grad_{\exp_{e}\f1{\sqrt{n}t}y}F_p(\cdot,x)\right)\right)\right)-f(y)\right)\\&=\left\langle d_yf,\f{y}{t}\right\rangle -{\sqrt n}\langle d_yf, \d\grad_eF_p(\cdot,x)\rangle-\d\n dF_p(\cdot,x)\left(\grad_yf,\f{y}{t}\right)\\
&+\f{\d^2}{2}{\rm Hess}_yf\left(\grad_eF_p(\cdot,x)\otimes \grad_eF_p(\cdot,x)\right)+O\left(\f1{\sqrt n}\right)
\end{split}
\end{equation}
uniformly in $y\in B$, $x\in {\rm supp}(\mu)$, $t\in [\e,T]$.
\end{lemma}
\begin{proof}
Let $x\in {\rm supp}(\mu)$, $y\in T_eM$, $u,v\in \RR$ sufficiently close to~$0$,   and $\di q=\exp_e\left(\f{uy}{t}\right)$.
 For $s\in [0,1]$ denote by $a\mapsto c(a,s,u,v)$ the geodesic with endpoints $c(0,s,u,v)=e$ and $$ c(1,s,u,v)=\exp_{\exp_e\left(\f{uy}{t}\right)}\left(-vs\grad_{\exp_e\left(\f{uy}{t}\right)}F_p(\cdot ,x)\right):$$
$$
c(a,s,u,v)=\exp_e\left\{a\exp_e^{-1}\left[\exp_{\exp_e\left(\f{uy}{t}\right)}\left(-sv\grad_{\exp_e\left(\f{uy}{t}\right)}F_p(\cdot, x)\right)\right]\right\}.
$$
This is a $C^2$ function of $(a,s,u,v)\in[0,1]^2\times (-\eta,\eta)^2$, $\eta$ sufficiently small. It also depends in a $C^2$ way of $x$ and $y$.
 Letting $\di c(a,s)=c\left(a,s,\f{1}{\sqrt{n}}, \f{\d}{nt+1}\right)$, we have 
$$
\exp_{e}^{-1}\left(\exp_{\exp_e\f1{\sqrt{n}t}y}\left(-\f{\d}{nt+1}\grad_{\exp_{e}\f1{\sqrt{n}t}y}F_p(\cdot,x)\right)\right)=\partial_ac(0,1).
$$
So we need a Taylor expansion up to order $n^{-1}$ of $\di \f{nt+1}{\sqrt{n}}\partial_ac(0,1)$.

We have $c(a,s,0,1)=\exp_e\left(-as\grad_eF_p(\cdot, x)\right)$ and this implies 
$$
\partial_s^2\partial_a c(0,s,0,1)=0, \quad\hbox{so}\quad\partial_s^2\partial_a c(0,s,u,1)=O(u).
$$
On the other hand the identities $c(a,s,u,v)=c(a,sv,u,1)$ yields  $\partial_s^2\partial_a c(a,s,u,v)=v^2\partial_s^2\partial_a c(a,s,u,1)$, so we obtain
$$
\partial_s^2\partial_a c(0,s,u,v)=O(uv^2)
$$
and this yields 
$$
\partial_s^2\partial_a c(0,s)=O(n^{-5/2}),
$$
uniformly in $s,x,y, t$.
But since 
$$
\left\|\partial_ac(0,1)-\partial_ac(0,0)-\partial_s\partial_ac(0,0)\right\|\le \f12\sup_{s\in [0,1]}\|\partial_s^2\partial_ac(0,s)\|
$$
we only need to estimate $\partial_ac(0,0)$ and $\partial_s\partial_ac(0,0)$.

 Denoting by $J(a)$ the Jacobi field $\partial_sc(a,0)$ we have
$$
\f{nt+1}{\sqrt n}\partial_ac(0,1)=\f{nt+1}{\sqrt n}\partial_ac(0,0)+\f{nt+1}{\sqrt n}\dot J(0)+O\left(\f1{n^{2}}\right).
$$
On the other hand
$$
\f{nt+1}{\sqrt n}\partial_ac(0,0)=\f{nt+1}{\sqrt n}\f{y}{\sqrt{n}t}=y+\f{y}{nt}
$$
so it remains to estimate $\dot J(0)$.

The Jacobi field $a\mapsto J(a,u,v)$ with endpoints $J(0,u,v)=0_e$ and $$J(1,u,v)=-v\grad_{\exp_e\left(\f{uy}{t}\right)}F_p(\cdot,x)$$
satisfies 
$$
\n_a^2J(a,u,v)=-R(J(a,u,v), \partial_ac(a,0,u,v))\partial_ac(a,0,u,v)=O(u^2v).
$$
This implies that 
$$
\n_a^2J(a)=O(n^{-2}).
$$
Consequently, denoting by $P_{x_1,x_2} : T_{x_1}M\to T_{x_2}M$ the parallel transport along the minimal geodesic from $x_1$ to $x_2$ (whenever it is unique) we have
\begin{equation}\label{JF1}
P_{c(1,0),e}J(1)=J(0)+\dot J(0)+O(n^{-2})=\dot J(0)+O(n^{-2}).
\end{equation}
But we also have 
\begin{align*}
P_{c(1,0,u,v),e}J(1,u,v)&=P_{c(1,0,u,v),e}\left(-v\grad_{c(1,0,u,v)}F_p(\cdot, x)\right)\\
&=-v\grad_eF_p(\cdot,x)-v\n_{\partial_ac(0,0,u,v)}\grad_{\cdot}F_p(\cdot, x)+O(vu^2)\\
&=-v\grad_eF_p(\cdot,x)-v\n dF_p(\cdot,x)\left(\f{uy}{t},\cdot\right)^\sharp+O(vu^2)
\end{align*}
where we used $\partial_ac(0,0,u,v)=\f{uy}{t}$ and for vector fields $A,B$ on $TM$ and a $C^2$ function $f_1$ on $M$
\begin{align*}
 \langle \n_{A_e}\grad f_1, B_e\rangle&=A_e\langle \grad f_1,B_e\rangle-\langle \grad f_1,\n_{A_e}B\rangle\\
&=A_e\langle df_1,B_e\rangle-\langle df_1,\n_{A_e}B\rangle\\
&=\n df_1(A_e,B_e)
\end{align*}
which implies 
$$
\n_{A_e}\grad f_1=\n df_1(A_e,\cdot)^\sharp.
$$
We obtain 
$$
P_{c(1,0),e}J(1)=-\f{\d}{nt+1}\grad_eF_p(\cdot,x)-\f{\d}{\sqrt{n}(nt+1)}\n dF_p(\cdot,x)\left(\f{y}{t},\cdot\right)^\sharp+O(n^{-2}).
$$
Combining with~\eqref{JF1} this gives 
$$
\dot J(0)=-\f{\d}{nt+1}\grad_eF_p(\cdot, x)-\f{\d}{nt+1}\n dF_p(\cdot, x)\left(\f{y}{\sqrt{n}t},\cdot\right)^\sharp+O\left(\f{1}{n^2}\right).
$$
 So finally
\begin{equation}\label{3.19}
 \f{nt+1}{\sqrt n}\partial_ac(0,1)=y+\f{y}{nt}-\f{\d}{\sqrt n}\grad_eF_p(\cdot, x)-\d\n dF_p(\cdot, x)\left(\f{y}{nt},\cdot\right)^\sharp+O\left(n^{-3/2}\right).
\end{equation}
To get the final result we are left to make a Taylor expansion of $f$ up to order~$2$.
\end{proof}

Define the following quantities:
\begin{equation}
 \label{3.9}
b_n(z)=n\int_{\{|z'-z|\le 1\}}(z'-z)P^n(z,dz')
\end{equation}
and
\begin{equation}
 \label{3.10}
a_n(z)=n\int_{\{|z'-z|\le 1\}} (z'-z)\otimes (z'-z) P^n(z,dz').
\end{equation}
The following property holds:
\begin{lemma}
 \label{L3.1}
Assume that either $p\ge 2$ or $e$ does not belong to the support ${\rm supp}(\mu)$.
%Assume that there exists a neighbourhood $V_e$ of $e$ such that for all $x\in B (a,r)$ the application $F_p(\cdot, x)$ %is of class $C^2$ in $V_e$.
\begin{itemize}
 \item[(1)] For all $R>0$ and $\e>0$, there exists $n_0$ such that for all $n\ge n_0$ and $z\in [\e,T]\times
 B(0_e,R)$, where $B(0_e,R)$ is the open ball in $T_eM$ centered at
 the origin with radius $R$,
\begin{equation}
 \label{3.11}
\int 1_{\{|z'-z|>1\}}\,P^n(z,dz')=0.
\end{equation}
\item[(2)] For all $R>0$ and $\e>0$,
\begin{equation}
 \label{3.12}
\lim_{n\to\infty}\sup_{z\in [\e,T]\times B(0_e,R)} |b_n(z)-b(z)|=0
\end{equation}
with
\begin{equation}
 \label{3.13}
b(z)=\left(1,\f1tL_\d(y)\right)\quad\hbox{and}\quad L_\d(y)=y-\d\n dH(y,\cdot)^\sharp.
\end{equation}
\item[(3)] For all $R>0$ and $\e>0$,
\begin{equation}
 \label{3.14}
\lim_{n\to\infty}\sup_{z\in [\e,T]\times B(0_e,R)} |a_n(z)-a(z)|=0
\end{equation}
with
\begin{equation}
 \label{3.15}
a(z)=\d^2{\rm diag}(0,\G)\quad \hbox{and}\quad \G=\EE\left[\grad_eF_p(\cdot,P_1)\otimes \grad_eF_p(\cdot, P_1)\right].
\end{equation}
\end{itemize}

\end{lemma}

\begin{proof}
\begin{itemize}
\item[(1)]
 We use the notation $z=(t,y)$ and $z'=(t',y')$. We have
\begin{align*}
 &\int 1_{\{|z'-z|>1\}}\,P^n(z,dz')\\&=\int 1_{\{\max(|t'-t|,|y'-y|)>1\}}P^n(z,dz')\\
&=\int 1_{\{\max(\f1n,|y'-y|)>1\}}P^n(z,dz')\\
&=\PP\left[\left|\f{nt+1}{\sqrt{n}}\exp_{e}^{-1}\left(\exp_{\exp_e\f1{\sqrt{n}t}y}\left(-\f{\d}{nt+1}\grad_{\exp_{e}\f1{\sqrt{n}t}y}F_p(\cdot,P_{1})\right)\right)-y\right|>
1\right].
\end{align*}
On the other hand, since $F_p(\cdot, x)$ is of class $C^2$ in a
neighbourhood of~$e$, we have by~\eqref{3.19}
\begin{equation}
 \label{3.16}
\left|\f{nt+1}{\sqrt{n}}\exp_{e}^{-1}\left(\exp_{\exp_e\f1{\sqrt{n}t}y}\left(-\f{\d}{nt+1}\grad_{\exp_{e}\f1{\sqrt{n}t}y}F_p(\cdot,P_{1})\right)\right)-y\right|\le \f{C\d}{\sqrt n\e}
\end{equation}
for some constant $C>0$.

\item[(2)] Equation \eqref{3.11} implies that for $n\ge n_0$
\begin{align*}
 &b_n(z)\\&=n\int (z'-z)\, P^n(z,dz')\\
&=n\left(\f1n,\E\left[\f{nt+1}{\sqrt{n}}\exp_{e}^{-1}\left(\exp_{\exp_e\f{y}{\sqrt{n}t}}\left(-\f{\d}{nt+1}\grad_{\exp_{e}\f{y}{\sqrt{n}t}}F_p(\cdot,P_{1})\right)\right)\right]-y\right).
\end{align*}
We have by lemma~\ref{L3.2}
\begin{align*}
&n\left(\f{nt+1}{\sqrt{n}}\exp_{e}^{-1}\left(\exp_{\exp_e\f1{\sqrt{n}t}y}\left(-\f{\d}{nt+1}\grad_{\exp_{e}\f1{\sqrt{n}t}y}F_p(\cdot,P_{1})\right)\right)-y\right)\\&=\f1{t}y-\d\sqrt{n}\grad_eF_p(\cdot,P_{1})-\d\n dF_p(\cdot,
P_1)\left(\f1{t}y,\cdot\right)^\sharp +O\left(\f1{n^{1/2}}\right)
\end{align*}
a.s. uniformly in $n$, and since
$$
\EE\left[\d\sqrt{n}\grad_eF_p(\cdot,P_{1})\right]=0,
$$
 this implies that
$$
n\left(\E\left[\f{nt+1}{\sqrt{n}}\exp_{e}^{-1}\left(\exp_{\exp_e\f1{\sqrt{n}t}y}\left(-\f{\d}{nt+1}\grad_{\exp_{e}\f1{\sqrt{n}t}y}F_p(\cdot,P_{1})\right)\right)\right]-y\right)
$$
converges to
\begin{equation}
 \label{3.17}
\f1ty-\EE\left[\d\n dF_p(\cdot, P_1)\left(\f1ty,\cdot\right)^\sharp\right]=\f1ty-\d\n dH_p\left(\f1ty,\cdot\right)^\sharp.
\end{equation}
Moreover the convergence is uniform in $z\in [\e,T]\times B(0_e,R)$,
so this yields~\eqref{3.12}.
\item[(3)]
In the same way, using lemma~\ref{L3.2},
\begin{align*}
& n\int(y'-y)\otimes (y'-y)\,P^n(z,dz')\\
&=\f1n\EE\left[\left(-\sqrt{n}\d\grad_eF_p(\cdot, P_1)\right)\otimes \left(-\sqrt{n}\d\grad_eF_p(\cdot, P_1)\right)\right]+o(1)\\
&=\d^2\EE\left[\grad_eF_p(\cdot, P_1)\otimes \grad_eF_p(\cdot,
P_1)\right]+o(1)
\end{align*}
uniformly in $z\in [\e,T]\times B(0_e,R)$, so this
yields~\eqref{3.14}.
\end{itemize}
\end{proof}

\begin{lemma}
 \label{L3.3}
Suppose that $\di t_n=\f{\d}{n}$ for some $\d>0$. For all $\d>C_{p,\mu,K}^{-1}$,
 \begin{equation}
  \label{3.20}
\sup_{n\ge 1}n\EE\left[\rho^2(e,X_n)\right]<\infty.
 \end{equation}
\end{lemma}
\begin{proof}

First consider the case $p\in[1,2)$.

 We know by \eqref{2.9} that there exists some constant $C(\b,r,p)$ such that
\begin{equation}\label{3.21}
\EE\left[\rho^2(e,X_{k+1})\right]\le
\EE\left[\rho^2(e,X_{k})\right]\exp\left(-C_{p,\mu,K}t_{k+1}\right)+C(\b,r,p)t_{k+1}^2.
\end{equation}
From this~\eqref{3.20} is a consequence of  Lemma~0.0.1 (case $\a>1$) in \cite{Nedic-Bertsekas:00}. We give the proof for completeness.
We deduce easily by induction that for all $k\ge k_0$,
\begin{equation}
 \label{3.22}
\begin{split}
&\EE\left[\rho^2(e,X_{k})\right]\\&\le\EE\left[\rho^2(e,X_{k_0})\right]\exp\left(-C_{p,\mu,K}\sum_{j=k_0+1}^kt_j\right)+C(\b,r,p)\sum_{i=k_0+1}^kt_i^2\exp\left(-C_{p,\mu,K}\sum_{j=i+1}^kt_j\right),
\end{split}
\end{equation}
where the convention $\sum_{j=k+1}^k t_j=0$ is used. With
$t_n=\f{\d}n$, the following inequality holds for all $i\ge k_0$ and
$k\ge i$:
\begin{equation}
 \label{3.23}
\sum_{j=i+1}^kt_j=\d\sum_{j=i+1}^k\f1j\ge
\d\int_{i+1}^{k+1}\f{dt}t\ge \d\ln\f{k+1}{i+1}.
\end{equation}
Hence,
 \begin{equation}
 \label{3.24}
\begin{split}
&\EE\left[\rho^2(e,X_{k})\right]\\&\le\EE\left[\rho^2(e,X_{k_0})\right]\left(\f{k_0+1}{k+1}\right)^{\d
C_{p,\mu,K}}+\f{\d^2C(\b,r,p)}{(k+1)^{\d
C_{p,\mu,K}}}\sum_{i=k_0+1}^k\f{(i+1)^{\d C_{p,\mu,K}}}{i^2}.
\end{split}
\end{equation}
For $\d C_{p,\mu,K}>1$ we have as $k\to\infty$
\begin{equation}
 \label{3.25}
\f{\d^2C(\b,r,p)}{(k+1)^{\d
C_{p,\mu,K}}}\sum_{i=k_0+1}^k\f{(i+1)^{\d C_{p,\mu,K}}}{i^2}\sim
\f{\d^2C(\b,r,p)}{(k+1)^{\d C_{p,\mu,K}}}\f{k^{\d C_{p,\mu,K}-1}}{\d
C_{p,\mu,K}-1}\sim \f{\d^2C(\b,r,p)}{\d C_{p,\mu,K}-1}k^{-1}
\end{equation}
and
$$
\EE\left[\rho^2(e,X_{k_0})\right]\left(\f{k_0+1}{k+1}\right)^{\d
C_{p,\mu,K}}=o(k^{-1}).
$$
This implies that the sequence $k\EE\left[\rho^2(e,X_{k})\right]$ is bounded.

Next consider the case $p\ge 2$.

Now we have by \eqref{2.9bis} that 
\begin{equation}\label{3.21bis}
\EE\left[H_p(X_{k+1})-H_p(e)\right]\le
\EE\left[H_p(X_{k})-H_p(e)\right]\exp\left(-C_{p,\mu,K}t_{k+1}\right)+C(\b,r,p)t_{k+1}^2.
\end{equation}
From this, arguing similarly, we obtain that the sequence $k\EE\left[H_p(X_{k})-H_p(e)\right]$ is bounded. We conclude with~\eqref{2.H3}.
\end{proof}

\begin{lemma}
 \label{L3.4}
Assume $\d >C_{p,\mu,K}^{-1}$ and that $H_p$ is $C^2$ in a neighbourhood of~$e$.
For all $0<\e<T$, the sequence of processes $\di \left(Y_{[nt]}^n\right)_{\e\le t\le T}$ is tight in $\DD([\e,T],\RR^d)$.
\end{lemma}
\begin{proof}
 Denote by $\di \left(\tilde Y_\e^n=\left(Y_{[nt]}^n\right)_{\e\le t\le T}\right)_{n\ge 1}$, the sequence of processes. We prove that from any subsequence $\di \left(\tilde Y_\e^{\phi(n)}\right)_{n\ge 1}$, we can extract a further subsequence $\di \left(\tilde Y_\e^{\psi(n)}\right)_{n\ge 1}$ that weakly converges in $\DD([\e,1],\RR^d)$.

Let us first prove that $\di \left(\tilde Y_\e^{\phi(n)}(\e)\right)_{n\ge 1}$ is bounded in $L^2$.
\[
\left\|\tilde Y_\e^{\phi(n)}(\e)\right\|_2^2=\f{[\phi(n)\e]^2}{\phi(n)}\EE\left[\rho^2(e,X_{[\phi(n)\e]})\right]\\
\le \e\sup_{n\ge 1}\left(n\EE\left[\rho^2(e,X_n)\right]\right)
\]
and the last term is bounded by lemma \ref{L3.3}.

Consequently $\di \left(\tilde Y_\e^{\phi(n)}(\e)\right)_{n\ge 1}$ is tight.
  So there is a subsequence $\di \left(\tilde Y_\e^{\psi(n)}(\e)\right)_{n\ge 1}$ that weakly converges in $T_eM$ to the distribution $\nu_\e$.  Thanks to Skorohod theorem which allows to realize it as an a.s. convergence and to lemma~\ref{L3.1} we can apply Theorem~11.2.3 of \cite{Stroock-Varadhan:79}, and we obtain that the sequence of processes $\di \left(\tilde Y_\e^{\psi(n)}\right)_{n\ge 1}$ weakly converges to a diffusion $(y_t)_{\e\le t\le T}$ with generator $G_\d(t)$ given by \eqref{3.26bis} and such that $y_\e$ has law $\nu_\e$. This achieves the proof of lemma~\ref{L3.4}.
\end{proof}

{\it Proof of Theorem \ref{T1}.}\
Let $\di \tilde Y^n=\left(Y_{[nt]}^n\right)_{0\le t\le T}$. It is sufficient to prove that any subsequence of $\di \left(\tilde Y^n\right)_{n\ge 1}$ has a further subsequence which converges in law to $(y_\d(t))_{0\le t\le T}$.
So let $\di\left(\tilde Y^{\phi(n)}\right)_{n\ge 1}$ a subsequence.  By lemma~\ref{L3.4} with $\e=1/m$ there exists a subsequence which converges in law on $[1/m,T]$. Then we extract a sequence indexed by $m$ of subsequence and take the diagonal subsequence   $\di \tilde Y^{\eta(n)}$. This subsequence converges in $\DD((0,T],\RR^d)$ to $(y'(t))_{t\in (0,T]}$. On the other hand, as  in the proof of lemma \ref{L3.4}, we have
$$
\|\tilde Y^{\eta(n)}(t)\|_2^2\le Ct
$$
for some $C>0$. So $\|\tilde Y^{\eta(n)}(t)\|_2^2\to 0$ as $t\to 0$, which in turn implies $\|y'(t)\|_2^2\to 0$ as $t\to 0$. The unicity statement in Proposition \ref{pr2} implies that $(y'(t))_{t\in (0,T]}$ and $(y_\d(t))_{t\in (0,T]}$ are equal in law.  This achieves the proof.
\CQFD

%%%%%%%%%%%%%%%%%%%%%%%%%%%%%%%%%%%%%%%%%%%%%%%%%%%%%%%%%%%%%%%%%%%%%%%%%
%
%   R E F E R E N C E S
%
%%%%%%%%%%%%%%%%%%%%%%%%%%%%%%%%%%%%%%%%%%%%%%%%%%%%%%%%%%%%%%%%%%%%%%%%%

\providecommand{\bysame}{\leavevmode\hbox to3em{\hrulefill}\thinspace}

\end{document}